\documentclass[twoside]{article}

\usepackage{amsmath,amsthm,amssymb}
\usepackage[utf8]{inputenc} 

\usepackage{mathptmx}

\usepackage[text={12.5cm,19cm},centering,paperwidth=17cm,paperheight=24cm]{geometry}

\usepackage[fontsize=10.4pt]{scrextend}

\def\titlerunning#1{\gdef\titrun{#1}}

\makeatletter
\def\author#1{\gdef\autrun{\def\and{\unskip, }#1}\gdef\@author{#1}}
\def\address#1{{\def\and{\\\hspace*{15.6pt}}\renewcommand{\thefootnote}{}\footnote{#1}}\markboth{\autrun}{\titrun}}
\makeatother

\def\email#1{email: \href{mailto:#1}{#1} }
\def\subjclass#1{\par\bigskip\noindent\textbf{Mathematics Subject Classification 2020.} #1}
\def\keywords#1{\par\smallskip\noindent\textbf{Keywords.} Periodic differential operators, homogenization, operator error estimates}

\newenvironment{dedication}{\itshape\center}{\par\medskip}
\newenvironment{acknowledgments}{\bigskip\small\noindent\textit{Acknowledgments.}}{\par}

\newtheorem{thm}{Theorem}[section]

\newtheorem{condition}[thm]{Condition}
\newtheorem{proposition}[thm]{Proposition}
\newtheorem{corollary}[thm]{Corollary}

\theoremstyle{definition}

\newtheorem*{rem}{Remark}

\numberwithin{equation}{section}

\usepackage[hyperfootnotes=false,colorlinks=true,allcolors=blue]{hyperref}

\begin{document}

\titlerunning{Homogenization of the higher-order Schr{\"o}dinger-type equations}

\title{\textbf{Homogenization of the higher-order Schr{\"o}dinger-type equations  
\\with periodic coefficients}}

\author{Tatiana Suslina}

\date{}

\maketitle

\address{T. A. Suslina: St. Petersburg State University, Universitetskaya nab. 7/9,
St.~Petersburg, 199034, Russia; \email{t.suslina@spbu.ru}}

\begin{dedication}
To Ari Laptev on the occasion of his 70th birthday
\end{dedication}

\begin{abstract}
   In $L_2({\mathbb R}^d; {\mathbb C}^n)$, we consider a  
matrix strongly elliptic differential operator ${A}_\varepsilon$ of order $2p$, $p \geqslant 2$. 
The operator ${A}_\varepsilon$ is given by  
${A}_\varepsilon = b(\mathbf{D})^* g(\mathbf{x}/\varepsilon) b(\mathbf{D})$, $\varepsilon >0$,
 where $g(\mathbf{x})$ is a periodic, bounded, and positive definite matrix-valued function, and $b(\mathbf{D})$ 
 is a homogeneous differential operator of order $p$. We prove that, for fixed $\tau \in {\mathbb R}$ and $\varepsilon \to 0$,  the operator exponential 
 $e^{-i \tau {A}_\varepsilon}$ converges to $e^{-i \tau {A}^0}$ in the norm of operators acting from the Sobolev space $H^s({\mathbb R}^d; {\mathbb C}^n)$ (with a suitable $s$) into $L_2({\mathbb R}^d; {\mathbb C}^n)$. 
 Here $A^0$ is the effective operator. Sharp-order error estimate is obtained.  The results are applied to homogenization of the Cauchy problem 
 for the Schr{\"o}dinger-type \hbox{equation} $i \partial_\tau {\mathbf u}_\varepsilon = {A}_\varepsilon {\mathbf u}_\varepsilon + {\mathbf F}$, ${\mathbf u}_\varepsilon\vert_{\tau=0} = \boldsymbol{\phi}$.
 
\subjclass{Primary 35B27}
\keywords{Homogenization, periodic differential operators, operator error estimates}
\end{abstract}

\section*{Introduction}

The paper concerns homogenization theory of periodic differential operators (DOs). First of all, we mention the books
\cite{BeLP}, \cite{ZhKO}. 

\subsection{Operator error estimates} 

In a series of papers \cite{BSu1,BSu2,BSu3} by Birman and Suslina, an operator-theoretic  (spectral) approach to homogenization problems was developed. In $L_2({\mathbb R}^d; {\mathbb C}^n)$, a wide class of 
matrix strongly elliptic second order DOs ${A}_\varepsilon$ was studied. 
The operator ${A}_\varepsilon$ is given by  
${A}_\varepsilon = b(\mathbf{D})^* g(\mathbf{x}/\varepsilon) b(\mathbf{D})$, $\varepsilon >0$,
 where $g(\mathbf{x})$ is a bounded and positive definite \hbox{$(m\times m)$-matrix-valued} function
 periodic with respect to some lattice \hbox{$\Gamma \subset {\mathbb R}^d$}, and $b(\mathbf{D}) = \sum_{l=1}^d b_l D_l$ is a first order DO. Here $b_l$ are constant $(m \times n)$-matrices. It is assumed that $m \geqslant n $ and the symbol  $b(\boldsymbol{\xi})$ has maximal rank.

In \cite{BSu1}, it was shown that the resolvent  $({A}_\varepsilon +I)^{-1}$ converges in the  
operator norm in $L_2({\mathbb R}^d; {\mathbb C}^n)$ to the resolvent of an effective operator ${A}^0$, and 
\begin{equation}
\label{est_A_eps}
 \bigl\| ({A}_\varepsilon +I)^{-1} - ({A}^0+I)^{-1} \bigr\|_{L_2(\mathbb{R}^d) \to L_2(\mathbb{R}^d)} \leqslant C \varepsilon.
\end{equation}
The effective operator is given by ${A}^0= b(\mathbf{D})^* g^0  b(\mathbf{D})$, where $g^0$ is a constant positive matrix called the \textit{effective} matrix. In \cite{Su1}, a similar result was obtained for the parabolic semigroup:
\begin{equation}
\label{parab_est_A_eps}
 \bigl\| e^{- \tau {A}_\varepsilon} - e^{-\tau {A}^0} \bigr\|_{L_2(\mathbb{R}^d)\to L_2(\mathbb{R}^d)} \leqslant C(\tau) \varepsilon,\quad \tau >0.
\end{equation}
Estimates  \eqref{est_A_eps} and \eqref{parab_est_A_eps} are order-sharp. Such inequalities are called  \textit{operator error estimates} in homogenization theory.

A different approach to operator error estimates (the shift method) was developed by Zhikov and Pastukhova. In \cite{Zh,ZhPas1,ZhPas2}, estimates 
\eqref{est_A_eps}, \eqref{parab_est_A_eps} were obtained for the operators of acoustics and elasticity. Further results were discussed in a survey \cite{ZhPas3}.

The operator error estimates for the 
Schr{\"o}dinger-type  and hyperbolic equations were studied in  \cite{BSu4} and in the recent works  \cite{D,DSu1,DSu2,M2,Su4}. 
In operator terms, the behavior of the operator-valued functions  
 $e^{-i \tau {A}_\varepsilon}$,  $\cos (\tau {A}_\varepsilon^{1/2})$, 
${A}_\varepsilon^{-1/2} \sin (\tau {A}_\varepsilon^{1/2})$, \hbox{$\tau \in \mathbb{R}$}, was investigated. It turned out that the nature of the results differed from the case of elliptic and parabolic equations: the type of the operator norm must be changed. 

Let us dwell on the case of the operator exponential $e^{-i \tau {A}_\varepsilon}$. In \cite{BSu4}, the following 
sharp-order estimate was proved:
 \begin{equation}
\label{est_cos_A_eps}
\bigl\| e^{-i \tau {A}_\varepsilon}  - e^{-i \tau {A}^0} \bigr\|_{H^3(\mathbb{R}^d)\to L_2(\mathbb{R}^d)} 
\leqslant C(1+  |\tau|) \varepsilon.
\end{equation}
In \cite{Su4,D}, it was shown that in the general case the result
\eqref{est_cos_A_eps} is sharp both regarding the type of the operator norm and regarding the dependence of the estimate on $\tau$ (it is impossible to replace $(1+|\tau|)$ on the right by  
$(1+|\tau|)^\alpha$ with $\alpha<1$). On the other hand, under some additional assumptions the result admits improvement:
\begin{equation}
\label{usilenie_intro}
 \bigl\| e^{-i \tau {A}_\varepsilon}  - e^{-i \tau {A}^0} \bigr\|_{H^{2}(\mathbb{R}^d) \to L_2(\mathbb{R}^d)}  \leqslant C (1 + |\tau|)^{1/2} \varepsilon.
 \end{equation}

The operator-theoretic approach was applied to the higher-order operators $A_\varepsilon$ in 
\cite{Ven, KuSu}. It was assumed that the operator ${A}_\varepsilon$ is given by  
\begin{equation}
\label{Aeps_operator}
{A}_\varepsilon = b(\mathbf{D})^* g(\mathbf{x}/\varepsilon) b(\mathbf{D}),\quad 
\operatorname{ord} b(\mathbf{D})=p \geqslant 2,\quad \varepsilon >0,
\end{equation}
 where $g(\mathbf{x})$ is a periodic, bounded, and positive definite $(m\times m)$-matrix-valued function, and $b(\mathbf{D}) = \sum_{|\beta|=p} b_\beta D^\beta$. Here $b_\beta$ are constant $(m \times n)$-matrices. It is assumed that $m \geqslant n $ and the symbol  $b(\boldsymbol{\xi})$ has maximal rank.
In  \cite{Ven, KuSu}, an estimate of the form \eqref{est_A_eps} for such operators $A_\varepsilon$ was obtained.
A more accurate approximation for the resolvent of $A_\varepsilon$ was found recently in  \cite{SlSu1, SlSu2}. 
The shift method was applied to homogenization of the elliptic higher-order operators in the papers 
 \cite{P1}, \cite{P2} by Pastukhova.

\subsection{Main results} In the present paper, the behavior of the operator exponential $e^{-i \tau {A}_\varepsilon}$
for the operator ${A}_\varepsilon$ of order $2p$ given by \eqref{Aeps_operator} is studied.
Our main result is the following estimate:
\begin{equation}
\label{main_res}
 \bigl\| e^{-i \tau {A}_\varepsilon}  - e^{-i \tau {A}^0} \bigr\|_{H^{2p+1}(\mathbb{R}^d) \to L_2(\mathbb{R}^d)}  \leqslant C (1 + |\tau|) \varepsilon.
 \end{equation}
 Here $A^0 = b(\mathbf{D})^* g^0 b(\mathbf{D})$ is the effective operator. 
 By the interpolation with the obvious estimate $\|e^{-i \tau {A}_\varepsilon}  - e^{-i \tau {A}^0}\|_{L_2 \to L_2}\leqslant 2$, we also obtain ``intermediate'' results: 
\begin{equation*}
 \bigl\| e^{-i \tau {A}_\varepsilon}  - e^{-i \tau {A}^0} \bigr\|_{H^{s}(\mathbb{R}^d) \to L_2(\mathbb{R}^d)}  \leqslant 
 C(s) (1 + |\tau|)^{s/(2p+1)} \varepsilon^{s/(2p+1)}, \quad 0 \leqslant s \leqslant 2p+1.
 \end{equation*}

Under some additional assumptions formulated in terms of the spectral characteristics of $A=b(\mathbf{D})^* g(\mathbf{x}) b(\mathbf{D})$ near the bottom of the spectrum, it is proved that
\begin{equation}
\label{main_res_improved2}
 \bigl\| e^{-i \tau {A}_\varepsilon}  - e^{-i \tau {A}^0} \bigr\|_{H^{2p+2}(\mathbb{R}^d) \to L_2(\mathbb{R}^d)}  \leqslant C (1 + |\tau|) \varepsilon^2.
 \end{equation}
 This means that the difference $e^{-i \tau {A}_\varepsilon}  - e^{-i \tau {A}^0}$ is of order $O(\varepsilon^2)$
 in a suitable norm.  It should be noted that the imposed additional assumptions are valid automatically for a scalar operator $A_\varepsilon$ (i.~e., $n=1$) with real-valued coefficients.
By the interpolation, we deduce
\begin{equation*}
 \bigl\| e^{-i \tau {A}_\varepsilon}  - e^{-i \tau {A}^0} \bigr\|_{H^{s}(\mathbb{R}^d) \to L_2(\mathbb{R}^d)}  \leqslant 
 C(s) (1 + |\tau|)^{s/(2p+2)} \varepsilon^{s/(p+1)},  \quad 0 \leqslant s \leqslant 2p+2.
 \end{equation*}
 In particular, for $s=p+1$ we have
\begin{equation}
\label{main_res_improved}
 \bigl\| e^{-i \tau {A}_\varepsilon}  - e^{-i \tau {A}^0} \bigr\|_{H^{p+1}(\mathbb{R}^d) \to L_2(\mathbb{R}^d)}  \leqslant C (1 + |\tau|)^{1/2} \varepsilon.
 \end{equation}
 This improves  \eqref{main_res} regarding both the type of the norm and 
 the dependence  on $\tau$.

 We stress that, for the second order operators ${A}_\varepsilon$, there is an analog of \eqref{main_res_improved} (cf. \eqref{usilenie_intro}), but there is no analog of \eqref{main_res_improved2}. 
 
 The above results  are applied to homogenization of the Cauchy problem for the 
 Schr{\"o}dinger-type equation
 $$
 i \,\partial_\tau \mathbf{u}_\varepsilon(\mathbf{x}, \tau) = A_\varepsilon 
\mathbf{u}_\varepsilon(\mathbf{x}, \tau) + \mathbf{F}(\mathbf{x}, \tau),\quad
  \mathbf{u}_\varepsilon(\mathbf{x}, 0) = \boldsymbol{\phi} (\mathbf{x}).
 $$

\subsection{Method} We rely on the operator-theoretic approach. 
By the scaling transformation,  the problem is reduced to the study of the operator $e^{-i \tau \varepsilon^{-2p} A}$.
Next, using the Floquet-Bloch theory, we expand $A$ in the direct integral of the operators 
$A(\mathbf{k})$ acting in $L_2(\Omega;\mathbb{C}^n)$ and given by 
$b(\mathbf{D} + \mathbf{k})^* g(\mathbf{x}) b(\mathbf{D}+ \mathbf{k})$ with periodic boundary conditions; here $\Omega$ is the cell of the lattice $\Gamma$. Since $A(\mathbf{k})$ is an analytic operator family with compact \hbox{resolvent}, it can be studied by means of the analytic perturbation theory with respect to the one-dimensional parameter $t=|\mathbf{k}|$. 
It turns out that only the spectral characteristics of $A(\mathbf{k})$ near the bottom of the spectrum are responsible for  homogenization. It is convenient to study the family $A(\mathbf{k})$ in the framework of 
an abstract operator-theoretic scheme. 

\subsection{Plan of the paper} The paper consists of  five sections. In Section 1, the abstract operator-theoretic method is developed.  In Section 2, we introduce the class of operators $A$ acting in $L_2(\mathbb{R}^d;\mathbb{C}^n)$ and describe the direct integral expansion for $A$. Section 3 is devoted to application of the abstract results to the operator family $A(\mathbf{k})$. In Section 4, using the results of Section 3, we obtain approximations for the operator exponential of $A$. Section 5 is devoted to homogenization problems: we deduce approximations for 
the exponential $e^{-i \tau {A}_\varepsilon}$ from the results of Section 4 and apply them to find approximations for the solutions of the Cauchy problem. 

\subsection{Notation} 
Let $\mathfrak{H}$ and $\mathfrak{H}_*$ be complex separable Hilbert spaces. By $(\cdot,\cdot)_{\mathfrak{H}}$ and $\Vert \cdot \Vert _{\mathfrak{H}}$ we denote the inner product and the norm in $\mathfrak{H}$, respectively;
the symbol $\Vert \cdot \Vert _{\mathfrak{H}\rightarrow \mathfrak{H}_*}$ denotes the norm of a linear continuous operator acting from $\mathfrak{H}$ to $\mathfrak{H}_*$. Sometimes, we omit the indices.
If $\mathfrak{N}$ is a subspace of $\mathfrak{H}$, then ${\mathfrak N}^\perp$ denotes  its
 orthogonal complement. If $A$ is a closed linear operator in $\mathfrak{H}$, its domain and kernel are denoted by 
 $\operatorname{Dom} A$ and $\operatorname{Ker} A$, respectively; $\sigma(A)$ stands for 
 the spectrum of $A$.

The inner product and the norm in $\mathbb{C}^n$ are denoted by $\langle \cdot ,\cdot \rangle$ and $\vert \cdot \vert$, respectively, $\mathbf{1}_n = \mathbf{1}$ is the unit $(n\times n)$-matrix.
We denote  $\mathbf{x}=(x_1,\dots,x_d)\in \mathbb{R}^d$, $iD_j=\partial /\partial x_j$, $j=1,\dots,d$, $\mathbf{D}=-i\nabla =(D_1,\dots,D_d)$.

The class $L_2$ of $\mathbb{C}^n$-valued functions in a domain $\mathcal{O}\subset \mathbb{R}^d$  is denoted by $L_2(\mathcal{O};\mathbb{C}^n)$. The Sobolev classes of  $\mathbb{C}^n$-valued functions in a domain  $\mathcal{O}$ are denoted by $H^s(\mathcal{O};\mathbb{C}^n)$. For $n=1$, we write simply $L_2(\mathcal{O})$, $H^s(\mathcal{O})$, but sometimes we use such simple notation also for the spaces of vector-valued or matrix-valued functions.

\section{Abstract operator-theoretic scheme\label{abstr_scheme}}

\subsection{Polynomial nonnegative operator pencils\label{sec_abstract_scheme_1.1}} 
Let $\mathfrak{H}$ and $\mathfrak{H}_*$ be complex separable Hilbert spaces. 
Let $X(t)$ be a family of operators (a polynomial pencil) of the form
\begin{equation*}
X(t)=\sum_{j=0}^{p}X_{j}t^{j},\ \ t\in\mathbb{R},\ \ p\in\mathbb{N},\ \ p\geqslant 2.
\end{equation*}
The operators $X(t)$, $X_{j}$ act from $\mathfrak{H}$ into $\mathfrak{H}_*$. It is assumed that the operator $X_{0}$ is densely defined and closed, while the operator $X_{p}$ is defined on the whole $\mathfrak{H}$ and bounded. 
In addition, we impose the following conditions.

\begin{condition}\label{sup1}
For any $j=0,\dots,p$ and $t\in\mathbb{R}$, we have 
\begin{equation*}
\operatorname{Dom}X(t)=\operatorname{Dom}X_{0}\subset \operatorname{Dom}X_{j}
\subset \operatorname{Dom}X_{p}=\mathfrak{H}.
\end{equation*}
\end{condition}

\begin{condition}\label{sup2}
For any $j=0,\dots,p-1$ and $u\in \operatorname{Dom}X_{0}$ we have 
\begin{equation*}
\|X_{j}u\|_{\mathfrak{H}_*}\leqslant C_{0}\|X_{0}u\|_{\mathfrak{H}_*},
\end{equation*}
where a constant $C_{0}\geqslant 1$ is independent of $j$ and $u$.
\end{condition}

Under the above assumptions, the operator $X(t)$ is closed for $|t|\leqslant (2 (p-1)C_{0})^{-1}$. 

\emph{Our main object} is the following family of nonnegative selfadjoint operators in $\mathfrak{H}$:
\begin{equation*}
A(t)=X(t)^* X(t),\ \ t\in\mathbb{R},\ \  |t|\leqslant (2 (p-1)C_{0})^{-1}.
\end{equation*}
Denote $A(0)=X_{0}^{*}X_{0}=:A_{0}$,  
$\mathfrak{N}:= \operatorname{Ker}A_{0}= \operatorname{Ker}X_{0}$, and 
$\mathfrak{N}_* := \operatorname{Ker}X_{0}^*$.
Let $P$ be the orthogonal projection of $\mathfrak{H}$ onto the subspace $\mathfrak{N}$, and let 
 $P_*$ be the orthogonal projection of $\mathfrak{H}_*$ onto the subspace $\mathfrak{N}_*$.
 
\begin{condition}\label{sup3}
Suppose that the point $\lambda_{0}=0$ is an isolated point of the spectrum of $A_{0}$, and 
$n:=\operatorname{dim}\mathfrak{N}<\infty$, $n \leqslant n_* := \operatorname{dim}\mathfrak{N}_* \leqslant \infty.$
\end{condition}

By $d^{0}$ we denote the distance from the point $\lambda_{0}=0$ to $\sigma(A_{0})\setminus\{\lambda_{0}\}$. 
Let $F(t,h)$ be the spectral projection of the operator $A(t)$ corresponding to the interval
$[0,h]$. We fix a positive number $\delta\leqslant \min\{d^{0}/36,1/4\}$ and choose a number $t_0 >0$ such that 
\begin{equation}\label{1.5z}
t_{0} \leqslant \delta^{1/2}C_1^{-1}, \quad \text{where}\ \ C_1 = \max \{ (p-1) C_0, \|X_p\| \}.
\end{equation}
Note that $t_0 \leqslant 1/2$. The operator $X(t)$ is automatically closed for $|t|\leqslant t_{0}$, because 
$t_{0}\leqslant (2 (p-1)C_{0})^{-1}$.
According to \cite[Proposition 3.10]{Ven}, for  $|t|\leqslant t_{0}$ we have 
\begin{equation*}
F(t,\delta)=F(t,3\delta), \quad \operatorname{rank}F(t,\delta)=n.
\end{equation*}
This means that, for $|t|\leqslant t_{0}$,  
the operator $A(t)$ has exactly $n$ eigenvalues counting their multiplicities on the interval $[0,\delta]$,  and the interval $(\delta,3\delta)$ is free of the spectrum.
We write $F(t):=F(t,\delta)$.

\subsection{Operators $Z$, $R$, and $S$\label{secZ,R,S}}

Let $\mathcal{D} = \operatorname{Dom}X_{0} \cap \mathfrak{N}^\perp$.
Obviously, $\mathcal{D}$ is a Hilbert space with the inner product 
$(f_1, f_2)_{\mathcal{D}} = \left( X_0 f_1, X_0 f_2 \right)_{\mathfrak{H}_*}$, $f_1, f_2 \in \mathcal{D}$.

Let $u \in \mathfrak{H}_*$. Consider the equation $X_0^* (X_0 \psi - u) =0$ for $\psi \in \mathcal{D}$,
which is understood in the weak sense:
\begin{equation}
\label{1.4zz}
\left( X_0 \psi, X_0 \zeta \right)_{\mathfrak{H}_*} = \left( u, X_0 \zeta \right)_{\mathfrak{H}_*},
\quad \forall \zeta \in \mathcal{D}.
\end{equation}
The right-hand side of \eqref{1.4zz} is an antilinear continuous functional of $\zeta \in \mathcal{D}$. 
Hence, by the Riesz theorem, there exists a unique solution $\psi \in \mathcal{D}$, and
 $\| X_0 \psi \|_{\mathfrak{H}_*} \leqslant \| u \|_{\mathfrak{H}_*}$. 
 
 Now, let $\omega \in \mathfrak{N}$ and $u = - X_p \omega$. In this case, the solution of equation 
 \eqref{1.4zz} is denoted by $\psi(\omega)$. 
 We define a bounded linear operator $Z: \mathfrak{H} \to \mathcal{D}$ putting
\begin{equation*}
Z \omega = \psi(\omega), \quad \omega \in \mathfrak{N};\quad
Z v =0, \quad v \in \mathfrak{N}^\perp.
\end{equation*}
  
Next, we define the operator $R: \mathfrak{N} \to \mathfrak{N}_*$  by the relation
 $R \omega = X_0 \psi(\omega) + X_p \omega$, $\omega \in \mathfrak{N}$.
 Another representation for $R$ is given by 
$R = P_* X_p \vert_{\mathfrak{N}}$. The selfadjoint operator 
$
S = R^* R : \mathfrak{N} \to \mathfrak{N}
$
is called the \emph{spectral germ} of the operator family $A(t)$ at $t=0$. 
The germ $S$ is called \emph{non-degenerate} if $\operatorname{Ker} S = \{0\}$.

\subsection{Analytic branches of eigenvalues and eigenvectors of  $A(t)$} 
According to the analytic perturbation theory (see \cite{K} and also 
\cite{Ven}, \cite{KuSu}), for $|t|\leqslant t_0$ there exist real-analytic functions $\lambda_{j}(t)$ (the branches of  eigenvalues) and real-analytic $\mathfrak{H}$-valued functions $\varphi_{j}(t)$ (the branches of eigenvectors) such that 
\begin{equation*}
A(t)\varphi_{j}(t)=\lambda_{j}(t)\varphi_{j}(t),\ \ j=1,\dots,n,\
\ |t|\leqslant t_{0},
\end{equation*}
and the set $\{\varphi_{j}(t)\}_{j=1}^{n}$ forms an orthonormal basis in the space $F(t)\mathfrak{H}$ for  
\hbox{$|t|\leqslant t_{0}$}. For sufficiently small $t_{*}\in (0,t_{0}]$ we have the following convergent 
power series expansions (see \cite[Theorem 3.15]{Ven})
\begin{align}
\label{1.8z}
\lambda_{j}(t)&= \gamma_j t^{2p} +\mu_j t^{2p+1}+ \dots, \quad j=1,\dots,n,\ \ |t|\leqslant t_{*};
\\
\label{1.9z}
\varphi_{j}(t)&=\omega_j + t \varphi_{j}^{(1)} + \dots,\quad j=1,\dots,n,\ \ |t|\leqslant t_{*}.
\end{align}
We have $\gamma_j \geqslant 0$, $\mu_j \in {\mathbb R}$. The set $\omega_1, \dots, \omega_n$ forms an orthonormal basis in $\mathfrak{N}$.
The numbers $\gamma_j$ and the vectors $\omega_j$ are eigenvalues and eigenvectors of the spectral germ:
$S \omega_j = \gamma_j \omega_j$, $j=1,\dots,n.$

\subsection{Threshold approximations} 
The following statement was proved in \cite{Ven, KuSu}. 
Below different constants depending only on $p$ are denoted by $C(p)$.

\begin{proposition}
\label{prop1.1}
Suppose that Conditions {\rm\ref{sup1}}, {\rm\ref{sup2}}, and {\rm\ref{sup3}} are satisfied. 
Then for $|t| \leqslant t_0$ we have
\begin{align}
\label{1.5zz}
\| F(t) - P\| &\leqslant C_2 |t|, \quad C_2 = C(p) C_T,
\\
\label{1.6zz}
\| A(t)F(t) - t^{2p} S P\| &\leqslant C_3 |t|^{2p+1}, \quad C_3 = C(p) C_T^{2p+1}.
\end{align}
Here $C_T = p\, C_0^2 + \|X_p\|^2 \delta^{-1}$.
\end{proposition}

More accurate threshold approximations were found in 
the recent paper \cite[Theorem 3.2]{SlSu2}. 

\begin{proposition}
\label{prop1.1a}
Suppose that Conditions {\rm\ref{sup1}}, {\rm\ref{sup2}}, and {\rm\ref{sup3}} are satisfied. Let
\begin{equation}
\label{op_G}
G:= (RP)^*  X_1 Z  + (X_1 Z)^* RP.
\end{equation}
In terms of the expansions \eqref{1.8z}, \eqref{1.9z},
$$
G = \sum_{j=1}^n \mu_j (\cdot, \omega_j)_{\mathfrak{H}} \omega_j + 
\sum_{j=1}^n \gamma_j \left( (\cdot, \varphi_j^{(1)})_{\mathfrak{H}} \omega_j + 
(\cdot, \omega_j)_{\mathfrak{H}} \varphi_j^{(1)} \right).
$$ 
Then for 
$|t| \leqslant t_0$ we have
\begin{align}
\label{1.5yy}
\| F(t) - P\| &\leqslant C_4 |t|^p, \quad C_4 = C(p) C_T^p,
\\
\label{1.6yy}
\| A(t)F(t) - t^{2p} S P - t^{2p+1} G\| &\leqslant C_5 t^{2p+2}, \quad C_5 = C(p) C_T^{2p+2}.
\end{align}
\end{proposition}

\subsection{Approximation for $e^{-i\tau A(t)}$}

\begin{proposition}\label{prop_est1}
Denote
\begin{equation}
\label{J_def}
J(t,\tau) := \left(e^{-i\tau A(t)} - e^{-i \tau t^{2p} SP}\right) P.
\end{equation}
For $\tau \in \mathbb{R}$ and $|t| \leqslant t_0$ we have
\begin{equation}
\label{exp_est3}
\left\|  J(t,\tau) \right\| \leqslant 2 C_2 |t| + C_3 |\tau| |t|^{2p+1}.
\end{equation}
\end{proposition}

\begin{proof}
We put $E(t,\tau) := e^{-i\tau A(t)} F(t) - e^{-i \tau t^{2p} SP} P$,
$$
\Sigma(t,\tau) :=  e^{i \tau t^{2p} SP} E(t,\tau) = e^{i \tau t^{2p} SP} F(t) e^{-i\tau A(t)} - P.
$$
 Obviously,
\begin{equation}
\label{J_le_E}
\| J(t, \tau)\| \leqslant \| E(t,\tau)\| +  \| F(t) - P\|.
\end{equation}
We have $\Sigma(t,0) = F(t)-P$ and
$$
\Sigma'(t,\tau):= \frac{d \Sigma(t,\tau)}{d\tau} = i e^{i \tau t^{2p} SP} \left(  t^{2p} SP - A(t) F(t) \right) F(t) e^{-i\tau A(t)}.
$$
Since $\Sigma(t,\tau) = \Sigma(t,0) + \int_0^\tau \Sigma'(t,\rho)\, d\rho$, then
\begin{equation}
\label{E(tau)_est}
\| E(t,\tau)\| = \|\Sigma(t,\tau)\| \leqslant \| F(t) - P\| + |\tau| \left\|t^{2p} SP - A(t) F(t)\right\|.
\end{equation}
Combining this with  \eqref{1.5zz}, \eqref{1.6zz}, and \eqref{J_le_E}, we arrive at the required estimate   \eqref{exp_est3}.
\end{proof}

In the case where $G=0$, the result can be improved. 

\begin{proposition}\label{prop_est1_improved}
Let $G$ be the operator \eqref{op_G}. Suppose that $G=0$. Let $J(t,\tau)$ be the operator \eqref{J_def}.
Then for $\tau \in \mathbb{R}$ and $|t| \leqslant t_0$ we have
\begin{equation}
\label{exp_est3_improved}
\left\|  J(t,\tau) \right\| \leqslant 2 C_4 |t|^p + C_5 |\tau| t^{2p+2}.
\end{equation}
\end{proposition}
 
\begin{proof}
Estimate \eqref{exp_est3_improved} follows from \eqref{1.5yy}, \eqref{1.6yy},  \eqref{J_le_E}, \eqref{E(tau)_est} and 
the condition \hbox{$G=0$}.
\end{proof}

\subsection{Approximation for the operator $\exp (-i \tau \varepsilon^{-2p} A(t))$}
Let $\varepsilon >0$. We study the behavior of the operator $\exp (-i \tau \varepsilon^{-2p} A(t))$ 
for $\tau \in \mathbb{R}$ and small $\varepsilon$.  
Let us estimate the operator $J(t, \tau \varepsilon^{-2p} )$ multiplied by 
the ``smoothing factor'' $\varepsilon^{s} (t^2 + \varepsilon^{2})^{-s/2}$ 
with $s = 2p+1$. (In applications to DOs, such multiplying turns into smoothing.)

\begin{thm}\label{th_est3}
Let $J(t,\tau)$ be the operator \eqref{J_def}.
For $\tau \in \mathbb{R}$, $\varepsilon >0$, and $|t| \leqslant t_0$ we have
\begin{equation}
\label{exp_est1}
\left\|  J(t, \tau \varepsilon^{-2p}) \right\| \frac{\varepsilon^{2p+1}}{ (t^2 + \varepsilon^{2})^{p+1/2}}
\leqslant ( C_2  + C_3 |\tau|) \varepsilon.
\end{equation}
\end{thm}

\begin{proof}
From \eqref{exp_est3} with $\tau$ replaced by $\tau \varepsilon^{-2p}$ it follows that 
$$
\begin{aligned}
\left\|  J( t, \tau \varepsilon^{-2p}) \right\|  \frac{\varepsilon^{2p+1}}{ (t^2 + \varepsilon^{2})^{p+1/2}}
&\leqslant \left(2 C_2 |t| + C_3 |\tau| \varepsilon^{-2p} |t|^{2p+1}\right)
\frac{\varepsilon^{2p+1}}{ (t^2 + \varepsilon^{2})^{p+1/2}}
\\
&\leqslant (C_2 + C_3 |\tau|) \varepsilon. 
\end{aligned}
$$
\end{proof}

In the case where $G=0$, this result can be improved.

\begin{thm}\label{th_est3_improved}
Let $J(t,\tau)$ be the operator \eqref{J_def}.
Let $G$ be the operator \eqref{op_G}. Suppose that $G=0$. Then
for $\tau \in \mathbb{R}$, $\varepsilon >0$, and $|t| \leqslant t_0$ we have
\begin{equation}
\label{exp_est1_improved5}
\left\|  J(t, \tau \varepsilon^{-2p}) \right\| \frac{\varepsilon^{2p+2}}{ (t^2 + \varepsilon^{2})^{p+1}}
\leqslant \left( 2 C_4  t_0^{p-2} + C_5 |\tau| \right) \varepsilon^2.
\end{equation}
\end{thm}

\begin{proof}
Estimate \eqref{exp_est1_improved5} follows from \eqref{exp_est3_improved} with $\tau$ replaced by
$\tau \varepsilon^{-2p}$:
\begin{equation*}
\begin{aligned}
\left\|  J(t, \tau \varepsilon^{-2p}) \right\| \frac{\varepsilon^{2p+2}}{ (t^2 + \varepsilon^{2})^{p+1}}
&\leqslant \left( 2 C_4 |t|^p + C_5 |\tau| \varepsilon^{-2p} t^{2p+2} \right)  
\frac{\varepsilon^{2p+2}}{ (t^2 + \varepsilon^{2})^{p+1}}
\\
&\leqslant 2 C_4 t_0^{p-2} \varepsilon^2 + C_5  |\tau|  \varepsilon^{2}.
\end{aligned}
\end{equation*}
We took into account that $p \geqslant 2$ and $|t| \leqslant t_0$.
\end{proof}

\section{Periodic differential operators in $L_2(\mathbb{R}^d;\mathbb{C}^n)$}

\subsection{Lattices. The Gelfand transformation\label{Gelf_trans}}

Let $\mathbf{a}_1, \dots, \mathbf{a}_d$ be a basis in $\mathbb{R}^d$ generating the lattice $\Gamma$:
$$
\Gamma = \Bigl\{ \mathbf{a} \in \mathbb{R}^d: \  \mathbf{a} = \sum_{j=1}^d l_j {\mathbf a}_j,\  l_j \in \mathbb{Z} \Bigr\},
$$
and let $\Omega \subset \mathbb{R}^d$ be the elementary cell of $\Gamma$:
$$
\Omega = \Bigl\{ \mathbf{x} \in \mathbb{R}^d: \, \mathbf{x} = \sum_{j=1}^d \xi_j {\mathbf a}_j,\, 0< \xi_j < 1 \Bigr\}.
$$
The basis $\mathbf{b}_1, \dots, \mathbf{b}_d$ in $\mathbb{R}^d$ dual to $\mathbf{a}_1, \dots, \mathbf{a}_d$ is defined by the relations $\langle \mathbf{b}_i, \mathbf{a}_j \rangle = 2\pi \delta_{ij}$. This basis generates  the lattice 
$\widetilde{\Gamma}$ dual to $\Gamma$.  
Let $\widetilde{\Omega}$ be the central Brillouin zone of $\widetilde{\Gamma}$ given by
$$
\widetilde{\Omega} = \Bigl\{ \mathbf{k} \in \mathbb{R}^d: \, |\mathbf{k}| < |{\mathbf k} - {\mathbf b}|,\, 
0 \ne {\mathbf b} \in \widetilde{\Gamma} \Bigr\}.
$$
We use the notation $|\Omega| = \operatorname{meas} \Omega$, 
$| \widetilde{\Omega} | = \operatorname{meas} \widetilde{\Omega}$. Note that 
$|\Omega| | \widetilde{\Omega} | = (2\pi)^d$.

Let $r_0$ be the radius of the ball inscribed in $\operatorname{clos} \widetilde{\Omega}$. We have 
$
2 r_0 = \min_{0 \ne \mathbf{b} \in \widetilde{\Gamma}} | \mathbf{b} |.
$

Below, $\widetilde{H}^s(\Omega)$ \emph{stands for the subspace of all functions $f \in H^s(\Omega)$ such that 
the $\Gamma$-periodic extension of $f$ to $\mathbb{R}^d$ belongs to} $H^s_{\text{loc}}(\mathbb{R}^d)$.

Initially, the Gelfand transformation $\mathcal{U}$ is defined on the functions $\mathbf{v}$ belonging to the Schwartz class ${\mathcal S}({\mathbb R}^d; \mathbb{C}^n)$ by the formula 
$$
\widetilde{\mathbf v}(\mathbf{k}, \mathbf{x}) = ( {\mathcal U}{\mathbf v})(\mathbf{k}, \mathbf{x})
= | \widetilde{\Omega} |^{-1/2} \sum_{\mathbf{a} \in \Gamma} 
e^{-i \langle {\mathbf k}, {\mathbf x} + {\mathbf a}\rangle} {\mathbf v}(\mathbf{x}+ \mathbf{a}),
\quad {\mathbf x}\in \Omega,\quad {\mathbf k} \in \widetilde{\Omega}.
$$
Then $\mathcal{U}$ extends by continuity up to a \emph{unitary mapping}
\begin{equation}
\label{Gelfand}
{\mathcal U}: L_2(\mathbb{R}^d;{\mathbb C}^n) \to \int_{\widetilde{\Omega}} 
\oplus L_2 (\Omega; \mathbb{C}^n)\, d \mathbf{k} =: \mathcal{K}.
\end{equation}

The relation $\mathbf{v} \in H^p(\mathbb{R}^d;{\mathbb C}^n)$ is equivalent to 
 $\widetilde{\mathbf v} \in L_2( \widetilde{\Omega};\widetilde{H}^p(\Omega; {\mathbb C}^n))$.
Under the transformation $\mathcal{U}$, the operator of multiplication by a bounded $\Gamma$-periodic function
in $L_2(\mathbb{R}^d;{\mathbb C}^n)$ turns into multiplication by the same function on the fibers of the direct integral $\mathcal{K}$ (see \eqref{Gelfand}). The linear DO $b(\mathbf{D})$ of order $p$ applied to 
$\mathbf{v} \in H^p(\mathbb{R}^d;{\mathbb C}^n)$ turns into the operator $b(\mathbf{D}+ \mathbf{k})$
applied to $\widetilde{\mathbf{v}} ({\mathbf k},\cdot) \in \widetilde{H}^p(\Omega;{\mathbb C}^n)$.

\subsection{Factorized operators of order $2p$\label{sec_operator}}

In $L_2(\mathbb{R}^d;{\mathbb C}^n)$, we consider an operator $A$ formally given by 
the differential expression
\begin{equation}
\label{op_A}
A = b( \mathbf{D})^* g(\mathbf{x}) b( \mathbf{D}).
\end{equation}
Here $g(\mathbf{x})$ is a Hermitian $(m \times m)$-matrix-valued function, in general, with complex entries. 
It is assumed that $g(\mathbf{x})$ is $\Gamma$-periodic, bounded, and positive definite: 
\begin{equation}
\label{g_cond}
g, g^{-1} \in L_\infty(\mathbb{R}^d); \quad g(\mathbf{x}) >0.
\end{equation}
The operator $b(\mathbf{D})$ is given by
$b(\mathbf{D}) = \sum_{|\beta | =p} b_\beta \mathbf{D}^\beta$,
where $b_\beta$ are constant $(m \times n)$-matrices, in general, with complex entries.
It is assumed that $m \geqslant n$ and the symbol $b(\boldsymbol{\xi}) =  \sum_{|\beta| =p} 
b_\beta \boldsymbol{\xi}^\beta$ satisfies 
$\operatorname{rank} b(\boldsymbol{\xi}) =n$ for $0 \ne \boldsymbol{\xi} \in {\mathbb R}^d.$
This condition is equivalent to the estimates
\begin{equation}
\label{b^*b}
\alpha_0 {\mathbf 1}_n \leqslant b( \boldsymbol{\theta})^* b(\boldsymbol{\theta}) \leqslant \alpha_1 {\mathbf 1}_n, 
\quad \boldsymbol{\theta} \in \mathbb{S}^{d-1}, 
\quad 0 < \alpha_0 \leqslant \alpha_1 <\infty,
\end{equation}
with some constants $\alpha_0, \alpha_1$.

The precise definition of the operator $A$ is given in terms of the quadratic form
\begin{equation}
\label{form_a}
\mathfrak{a}[\mathbf{u}, \mathbf{u}] = \int_{\mathbb{R}^d} \langle g(\mathbf{x}) b(\mathbf{D}) \mathbf{u}(\mathbf{x}), b(\mathbf{D}) \mathbf{u}(\mathbf{x})\rangle \, d \mathbf{x}, \quad \mathbf{u} \in H^p(\mathbb{R}^d; \mathbb{C}^n).
\end{equation}
Using the Fourier transform and \eqref{g_cond}, \eqref{b^*b}, it is easy to check that
\begin{equation}
\label{form_a_est}
c_0 \int_{\mathbb{R}^d} | {\mathbf D}^p \mathbf{u}(\mathbf{x})|^2 \, d\mathbf{x} \leqslant
\mathfrak{a}[\mathbf{u}, \mathbf{u}] \leqslant c_1 \int_{\mathbb{R}^d} | {\mathbf D}^p \mathbf{u}(\mathbf{x})|^2 \, d\mathbf{x},
\quad \mathbf{u} \in H^p(\mathbb{R}^d; \mathbb{C}^n).
\end{equation}
Here  $| {\mathbf D}^p \mathbf{u}(\mathbf{x})|^2 := \sum_{|\beta|=p} | \mathbf{D}^\beta \mathbf{u}(\mathbf{x})|^2$. 
The constants $c_0, c_1$ are given by
\begin{equation}
\label{c0,c1}
c_0 = C(p) \alpha_0 \| g^{-1}\|^{-1}_{L_\infty},\quad c_1 = C(p) \alpha_1 \| g \|_{L_\infty}.
\end{equation}
Hence, the form \eqref{form_a} is closed and nonnegative.
By definition, $A$ is a selfadjoint operator in $L_2(\mathbb{R}^d;\mathbb{C}^n)$ generated by this form. 

Note that the operator $A$ can be written as $A= X^* X$, where 
$X: L_2(\mathbb{R}^d;\mathbb{C}^n) \to L_2(\mathbb{R}^d;\mathbb{C}^m)$ is a closed operator defined by 
$$
X  = g^{1/2} b(\mathbf{D}),\quad \operatorname{Dom} X = H^p(\mathbb{R}^d; \mathbb{C}^n).
$$

\subsection{Operators $A(\mathbf{k})$ in $L_2(\Omega;\mathbb{C}^n)$}
Let ${\mathbf k} \in \mathbb{R}^d$. 
In $L_2(\Omega; \mathbb{C}^n)$, we consider the quadratic form
\begin{equation}
\label{form_a(k)}
\mathfrak{a}(\mathbf{k})[\mathbf{u}, \mathbf{u}] = \int_{\Omega} \langle g(\mathbf{x}) 
b(\mathbf{D}+ \mathbf{k}) \mathbf{u}(\mathbf{x}), b(\mathbf{D}+ \mathbf{k}) \mathbf{u}(\mathbf{x})\rangle 
\, d \mathbf{x}, \quad \mathbf{u} \in \widetilde{H}^p(\Omega; \mathbb{C}^n).
\end{equation}
 Using the discrete Fourier transform and \eqref{g_cond}, \eqref{b^*b}, it is easy to check that
\begin{equation*}
c_0 \int_{\Omega} | ({\mathbf D} + \mathbf{k} )^p \mathbf{u}(\mathbf{x})|^2 \, d\mathbf{x} \leqslant
\mathfrak{a}(\mathbf{k})[\mathbf{u}, \mathbf{u}] \leqslant c_1 \int_{\Omega} | ({\mathbf D} + \mathbf{k})^p \mathbf{u}(\mathbf{x})|^2 \, d\mathbf{x},
\quad \mathbf{u} \in \widetilde{H}^p(\Omega; \mathbb{C}^n).
\end{equation*}
Here $c_0, c_1$ are the same constants as in \eqref{form_a_est}; see \eqref{c0,c1}.
Hence, the form \eqref{form_a(k)} is closed and nonnegative. A selfadjoint operator in $L_2(\Omega;\mathbb{C}^n)$
corresponding to this form is denoted by $A(\mathbf{k})$. Formally, we have
$
A(\mathbf{k}) = b(\mathbf{D}+ \mathbf{k})^* g(\mathbf{x}) b(\mathbf{D}+ \mathbf{k}).
$

Note that the operator $A(\mathbf{k})$ can be written as $A(\mathbf{k}) = X(\mathbf{k})^* X(\mathbf{k})$, where 
$X(\mathbf{k}): L_2(\Omega;\mathbb{C}^n) \to L_2(\Omega;\mathbb{C}^m)$ is a closed operator defined by 
$$
X(\mathbf{k})  = g^{1/2} b(\mathbf{D}+ \mathbf{k}),\quad \operatorname{Dom} X(\mathbf{k}) = 
\widetilde{H}^p(\Omega; \mathbb{C}^n).
$$

\subsection{Direct integral expansion for the operator $A$}
Using the Gelfand transform $\mathcal{U}$ defined in Subsection \ref{Gelf_trans}, we expand the operator $A$
in the direct integral of the operators $A(\mathbf{k})$.   Let $\mathbf{v}\in H^p(\mathbb{R}^d;\mathbb{C}^n)$
and let $\widetilde{\mathbf v}(\mathbf{k}, \mathbf{x}) = ({\mathcal U} \mathbf{v})(\mathbf{k}, \mathbf{x})$.
Then $\widetilde{\mathbf v} \in L_2(\widetilde{\Omega}; \widetilde{H}^p(\Omega;\mathbb{C}^n))$ and
\begin{equation}
\label{direct}
\mathfrak{a}[{\mathbf v},{\mathbf v} ] = \int_{\widetilde{\Omega}} \mathfrak{a}(\mathbf{k})
[ \widetilde{\mathbf v}(\mathbf{k}, \cdot), \widetilde{\mathbf v}(\mathbf{k}, \cdot) ] \, d \mathbf{k}.
\end{equation}
Conversely, if $\widetilde{\mathbf v} \in L_2(\widetilde{\Omega}; \widetilde{H}^p(\Omega;\mathbb{C}^n))$, then
$\mathbf{v} = {\mathcal U}^{-1}\widetilde{\mathbf v} \in H^p(\mathbb{R}^d;\mathbb{C}^n)$
and identity \eqref{direct} is fulfilled. This means that
\begin{equation}
\label{direct2}
{\mathcal U} A {\mathcal U}^{-1}  = \int_{\widetilde{\Omega}} \oplus A(\mathbf{k})\, d \mathbf{k}.
\end{equation}

\section{Application of the abstract results to $A(\mathbf{k})$}

\subsection{Incorporation of the operators $A(\mathbf{k})$ in the abstract scheme}

We shall apply the scheme of Section \ref{abstr_scheme},
putting $\mathfrak{H} = L_2(\Omega; \mathbb{C}^n)$ and $\mathfrak{H}_* = L_2(\Omega; \mathbb{C}^m)$.
 
We write $\mathbf{k}$ as ${\mathbf k} = t \boldsymbol{\theta}$, where $t = |\mathbf{k}|$ and
$ \boldsymbol{\theta} \in \mathbb{S}^{d-1}$.
The roles of $X(t)$ and $A(t)$ are played by the operators $X(\mathbf{k}) =: X(t, \boldsymbol{\theta})$ and $A(\mathbf{k}) =: A(t, \boldsymbol{\theta})$. They
depend on the one-dimensional parameter $t$ and the additional parameter $\boldsymbol{\theta}$,
which was absent in the abstract scheme. He have to take care about this and to prove estimates uniform in 
$\boldsymbol{\theta}$.

Let us check that all the assumptions of Section \ref{abstr_scheme} are fulfilled. We have
$$
X(\mathbf{k}) = g^{1/2}\sum_{|\beta|=p} b_\beta (\mathbf{D} + \mathbf{k})^\beta = 
 g^{1/2}\sum_{|\beta|=p} b_\beta \sum_{\gamma \leqslant \beta} C_\beta^\gamma t^{|\beta-\gamma|}
\boldsymbol{\theta}^{\beta-\gamma} \mathbf{D}^\gamma.
$$
Hence, the operator $X(\mathbf{k}) =: X(t, \boldsymbol{\theta})$ can be written as
$
X(t, \boldsymbol{\theta})= X_0 + \sum_{j=1}^p t^j X_j(\boldsymbol{\theta}),
$
where the operator
$X_0 = g^{1/2} b(\mathbf{D})$, $\operatorname{Dom} X_0 = \widetilde{H}^p(\Omega;\mathbb{C}^n),$
is closed, the operators $X_1(\boldsymbol{\theta}),\dots, X_{p-1}(\boldsymbol{\theta})$ are given by
\begin{equation}
\label{def_Xj}
X_j(\boldsymbol{\theta}) =  g^{1/2}\sum_{|\beta|=p} b_\beta \sum_{\gamma \leqslant \beta:
| \gamma| =p- j} C_\beta^\gamma \boldsymbol{\theta}^{\beta-\gamma} \mathbf{D}^\gamma,
 \quad \operatorname{Dom} X_j (\boldsymbol{\theta}) = \widetilde{H}^{p-j}(\Omega;\mathbb{C}^n),
\end{equation}
and the operator $X_p(\boldsymbol{\theta}) = g^{1/2} b(\boldsymbol{\theta})$ is bounded from $\mathfrak{H}$
to $\mathfrak{H}_*$. 

Obviously, Condition \ref{sup1} is satisfied.
Condition \ref{sup2} is also satisfied with 
\begin{equation}
\label{C0}
C_0 = C(d,p) \alpha_1^{1/2} \alpha_0^{-1/2} \|g\|_{L_\infty}^{1/2} \|g^{-1}\|_{L_\infty}^{1/2} (1+r_0^{-1})^{p-1},
\end{equation}
where $C(d,p)$ depends only on $d$ and $p$; see \cite[Proposition 5.2]{KuSu}.  

By \eqref{b^*b}, we obtain the uniform bound for the norm of $X_p(\boldsymbol{\theta})$:
\begin{equation}
\label{Xp_est}
\|X_p(\boldsymbol{\theta})\| \leqslant \alpha_1^{1/2} \|g\|^{1/2}_{L_\infty},
\quad \boldsymbol{\theta} \in {\mathbb S}^{d-1}.
\end{equation}

Let $\mathfrak{N} = \operatorname{Ker} A(0) = \operatorname{Ker} X_0$.
It is easy to check that $\mathfrak{N}$ consists of constant vector-valued functions (see \cite[Proposition 5.1]{KuSu}):
\begin{equation}
\label{N_got}
\mathfrak{N} = \{ \mathbf{u} \in L_2(\Omega; \mathbb{C}^n): \ \mathbf{u}(\mathbf{x}) = \mathbf{c} \in \mathbb{C}^n \}.
\end{equation}
So, $\operatorname{dim} \mathfrak{N} = n$. The orthogonal projection of $L_2(\Omega; {\mathbb C}^n)$
onto $\mathfrak N$ is the operator of the averaging over the cell:
\begin{equation}
\label{proj_P}
P \mathbf{u} = |\Omega|^{-1} \int_{\Omega} \mathbf{u} (\mathbf{x})\, d\mathbf{x},
\quad \mathbf{u} \in L_2(\Omega; {\mathbb C}^n).
\end{equation}

Let $\mathfrak{N}_* = \operatorname{Ker} X_0^*$ and $n_* = \operatorname{dim} \mathfrak{N}_*$.
The condition $m \geqslant n$ ensures that $n \leqslant n_*$. 
Moreover, either $n_* = \infty$ (if $m>n$), or $n_*=n$ (if $m=n$). See \cite[Section 5.1]{KuSu} for details.

Since the embedding of $\widetilde{H}^p(\Omega;\mathbb{C}^n)$ into $L_2(\Omega;\mathbb{C}^n)$
is compact, the spectrum of the operator $A(0)$ is discrete. The point $\lambda_0 =0$ is an isolated eigenvalue of $A(0)$ of multiplicity $n$; the corresponding eigenspace $\mathfrak{N}$ is given by \eqref{N_got}.
Thus, Condition \ref{sup3} is satisfied.

Let $d^0$ be the distance from the point $\lambda_0 =0$ to the rest of the spectrum of $A(0)$.
According to \cite[(5.17)]{KuSu}, 
\begin{equation}
\label{d0_est}
d^0 \geqslant \alpha_0 \| g^{-1}\|^{-1}_{L_\infty} (2 r_0)^{2p}.
\end{equation}
In Subsection \ref{sec_abstract_scheme_1.1} it was required to fix 
a positive number $\delta \leqslant \min \{d^0/36, 1/4\}$. Using  \eqref{d0_est}, we choose $\delta$ as follows:
\begin{equation}
\label{delta_fix}
\delta = \min \{  \alpha_0 \| g^{-1}\|^{-1}_{L_\infty} (2 r_0)^{2p}/36 , 1/4 \}.
\end{equation}

Next, the constant $C_1(\boldsymbol{\theta}) = \max \{(p-1)C_0, \|X_p(\boldsymbol{\theta})\|\}$
now depends on $\boldsymbol{\theta}$ (see \eqref{1.5z}). Using \eqref{Xp_est}, we see that 
$C_1(\boldsymbol{\theta}) \leqslant C_1$, where
\begin{equation}
\label{C1_fix}
C_1 = \max \{ (p-1) C_0, \alpha_1^{1/2} \|g\|^{1/2}_{L_\infty} \}.
\end{equation}
Here $C_0$ is given by \eqref{C0}.
According to \eqref{1.5z}, we fix a number $t_0 \leqslant \delta^{1/2} C_1(\boldsymbol{\theta})^{-1}$ as follows:
\begin{equation}
\label{t0_fix}
t_0 = \delta^{1/2} C_1^{-1},
\end{equation}
where $\delta$ and $C_1$ are defined by \eqref{delta_fix} and \eqref{C1_fix}, respectively.

\subsection{The operators $Z(\boldsymbol{\theta})$, $R(\boldsymbol{\theta})$, and $S(\boldsymbol{\theta})$}

For the operator family $A(t,\boldsymbol{\theta})$, the operators $Z$, $R$, and $S$  defined in Subsection \ref{secZ,R,S} in the abstract setting
depend on the parameter $\boldsymbol{\theta}$.

To describe these operators, we introduce the $(n \times m)$-matrix-valued function $\Lambda(\mathbf{x})$
which is a $\Gamma$-periodic solution of the following problem:
\begin{equation}
\label{Lambda_problem}
b(\mathbf{D})^* g(\mathbf{x}) \left( b(\mathbf{D}) \Lambda(\mathbf{x}) + {\mathbf 1}_m \right)=0,
\quad \int_\Omega \Lambda(\mathbf{x}) \, d{\mathbf x} =0.
\end{equation}
The equation is understood in the weak sense: for each ${\mathbf C} \in \mathbb{C}^m$
we have $\Lambda {\mathbf C} \in \widetilde{H}^p(\Omega; \mathbb{C}^n)$ and
$$
\int_\Omega  \langle g(\mathbf{x})  (b(\mathbf{D}) \Lambda(\mathbf{x}) \mathbf{C}+ \mathbf{C}), b(\mathbf{D}) \boldsymbol{\eta}(\mathbf{x}) \rangle \, d \mathbf{x} =0,\quad \boldsymbol{\eta} \in \widetilde{H}^p(\Omega; \mathbb{C}^n).
$$
Then (cf. \cite[Section 5.3]{KuSu})
\begin{equation}
\label{def_Z(theta)}
Z(\boldsymbol{\theta}) = [\Lambda] b(\boldsymbol{\theta}) P,
\end{equation}
where $[\Lambda]$ denotes the operator of multiplication by the matrix-valued function $\Lambda(\mathbf{x})$.
The operator $R(\boldsymbol{\theta})$ is given by
\begin{equation}
\label{def_R(theta)}
R(\boldsymbol{\theta}) = [g^{1/2}(b(\mathbf{D})\Lambda + \mathbf{1}_m)] b(\boldsymbol{\theta})\vert_{\mathfrak N}.
\end{equation}
Then (cf. \cite[Section 5.3]{KuSu}), the spectral germ $S(\boldsymbol{\theta})=R(\boldsymbol{\theta})^*R(\boldsymbol{\theta})$  acts in the subspace $\mathfrak{N}$ (see \eqref{N_got}) and is represented as
\begin{equation}
\label{germ_DO}
S(\boldsymbol{\theta}) = b(\boldsymbol{\theta})^* g^0 b(\boldsymbol{\theta}),
\quad \boldsymbol{\theta} \in \mathbb{S}^{d-1}.
\end{equation}
Here $g^0$ is the so called \emph{effective matrix} (of size $m \times m$) given by
$$
g^0 = |\Omega|^{-1} \int_{\Omega} \widetilde{g}(\mathbf{x}) \, d \mathbf{x},\quad
\widetilde{g}(\mathbf{x}) := g(\mathbf{x}) \left( b(\mathbf{D}) \Lambda(\mathbf{x}) + {\mathbf 1}_m \right).
$$
It turns out that the effective matrix $g^0$ is positive definite. So, the germ $S(\boldsymbol{\theta})$ is non-degenerate.
We mention some properties of $g^0$; see \cite[Propositions 5.3, 5.4]{KuSu}.

\begin{proposition}
\label{Voigt-Reuss}
Denote
$$
\overline{g} = |\Omega|^{-1} \int_\Omega g(\mathbf{x}) \, d\mathbf{x}, \quad
\underline{g} = \left(|\Omega|^{-1} \int_\Omega g(\mathbf{x})^{-1} \, d\mathbf{x} \right)^{-1}.
$$
The effective matrix $g^0$ satisfies the following estimates \emph{(}the Voigt--Reuss bracketing\emph{):}
$\underline{g} \leqslant g^0 \leqslant \overline{g}$.
In the case where $m=n$, we have $g^0 = \underline{g}$.
\end{proposition}

\begin{proposition}
\label{special_cases}
$1^\circ$. Let ${\mathbf g}_k(\mathbf{x})$, $k=1,\dots,m,$ be the columns of the matrix $g(\mathbf{x})$.
The relation $g^0 = \overline{g}$ is equivalent to the identities 
\begin{equation}
\label{g_verhnee}
b(\mathbf{D})^* {\mathbf g}_k(\mathbf{x})=0, \quad k=1,\dots,m.
\end{equation}

\noindent $2^\circ$. Let ${\mathbf l}_k(\mathbf{x})$, $k=1,\dots,m,$ be the columns of the matrix $g(\mathbf{x})^{-1}$.
The relation $g^0 = \underline{g}$ is equivalent to the representations  
\begin{equation}
\label{g_nizhnee}
{\mathbf l}_k(\mathbf{x}) = {\mathbf l}_k^0 + b(\mathbf{D}) \mathbf{v}_k(\mathbf{x}),
\quad {\mathbf l}_k^0 \in {\mathbb C}^m,\ \mathbf{v}_k  \in \widetilde{H}^p(\Omega;\mathbb{C}^n);
\quad  k=1,\dots,m.
\end{equation}
\end{proposition}

\begin{rem}
In the case where $g^0 = \underline{g}$, the matrix $\widetilde{g}(\mathbf{x})$ is constant: $\widetilde{g}(\mathbf x) = g^0 = \underline{g}$.
\end{rem}

\subsection{The effective operator}

By \eqref{germ_DO} and the homogeneity of the symbol $b(\mathbf{k})$, we have
\begin{equation}
\label{eff_op_symbol}
S(\mathbf{k}) := t^{2p} S(\boldsymbol{\theta}) = b(\mathbf{k})^* g^0 b(\mathbf{k}), 
\quad \mathbf{k} \in {\mathbb R}^d.
\end{equation}
Expression \eqref{eff_op_symbol} is the symbol of the DO
\begin{equation}
\label{eff_op}
A^0 =  b(\mathbf{D})^* g^0 b(\mathbf{D}), \quad \operatorname{Dom} A^0 = H^{2p}(\mathbb{R}^d;\mathbb{C}^n),
\end{equation}
which is called the \emph{effective operator} for $A$.

Let $A^0(\mathbf{k})$ be the operator family in $L_2(\Omega;\mathbb{C}^n)$ corresponding to the operator $A^0$.
Then $A^0(\mathbf{k})$ is given by the differential expression 
$ b(\mathbf{D} + \mathbf{k})^* g^0 b(\mathbf{D} + \mathbf{k})$ on the domain
$\widetilde{H}^{2p}(\Omega;\mathbb{C}^n)$. 
 
 By \eqref{proj_P} and \eqref{eff_op_symbol}, we have
\begin{equation}
\label{eff_op2}
S(\mathbf{k}) P = A^0(\mathbf{k}) P. 
\end{equation}

\subsection{The operator $G(\boldsymbol{\theta})$}

For $A(t,\boldsymbol{\theta})$, the operator $G$  defined by \eqref{op_G} 
 in the abstract setting depends on $\boldsymbol{\theta}$:
$$
G (\boldsymbol{\theta}) = (R(\boldsymbol{\theta}) P)^* X_1(\boldsymbol{\theta}) Z(\boldsymbol{\theta}) +
 (X_1(\boldsymbol{\theta}) Z(\boldsymbol{\theta}))^* R(\boldsymbol{\theta}) P.
$$
Let $B_1(\boldsymbol{\theta};\mathbf{D})$ be the DO of order $p-1$ such that 
$X_1(\boldsymbol{\theta}) = g^{1/2} B_1(\boldsymbol{\theta};\mathbf{D})$ (see \eqref{def_Xj}). Then
\begin{equation*}
B_1(\boldsymbol{\theta}; \mathbf{D}) = \sum_{|\beta|=p} b_\beta \sum_{\gamma \leqslant \beta:
| \gamma| =p- 1} C_\beta^\gamma \boldsymbol{\theta}^{\beta-\gamma} \mathbf{D}^\gamma.
\end{equation*}
Using \eqref{def_Z(theta)} and \eqref{def_R(theta)}, we obtain
\begin{equation}
\label{G(theta)=}
G (\boldsymbol{\theta}) = b(\boldsymbol{\theta})^* g^{(1)}(\boldsymbol{\theta})b(\boldsymbol{\theta}) P,
 \end{equation}
 where $g^{(1)} (\boldsymbol{\theta})$ is a Hermitian $(m \times m)$-matrix given by
\begin{equation}
\label{g1(theta)=}
g^{(1)} (\boldsymbol{\theta}) = |\Omega|^{-1}  \int_\Omega\left( \widetilde{g}(\mathbf{x})^*  B_1(\boldsymbol{\theta}; \mathbf{D}) \Lambda(\mathbf{x})  + (B_1(\boldsymbol{\theta}; \mathbf{D}) \Lambda(\mathbf{x}))^* \widetilde{g}(\mathbf{x}) \right) \, d \mathbf{x}.
 \end{equation}

 We distinguish some cases where the operator \eqref{G(theta)=} is equal to zero.
 
 \begin{proposition}
 \label{prop_cond_G=0}
$1^\circ$. Suppose that relations \eqref{g_verhnee} are satisfied. Then $\Lambda(\mathbf{x})=0$, whence 
$g^{(1)} (\boldsymbol{\theta})=0$ and $G (\boldsymbol{\theta})=0$.

\noindent$2^\circ$. Suppose that relations \eqref{g_nizhnee} are satisfied. Then  
$g^{(1)} (\boldsymbol{\theta})=0$ and $G (\boldsymbol{\theta})=0$.

\noindent$3^\circ$. Suppose that $n=1$ and the matrices $g(\mathbf{x})$, $b_\beta$, $|\beta|=p$, have real entries. Then $G (\boldsymbol{\theta})=0$ for any $\boldsymbol{\theta} \in \mathbb{S}^{d-1}$.
\end{proposition}

\begin{proof}
Obviously, if relations \eqref{g_verhnee} are satisfied, then the solution $\Lambda(\mathbf{x})$ of problem 
\eqref{Lambda_problem} is equal to zero. From \eqref{g1(theta)=} it follows that 
$g^{(1)} (\boldsymbol{\theta})=0$. Then, by \eqref{G(theta)=}, \hbox{$G(\boldsymbol{\theta})=0$}.

If relations \eqref{g_nizhnee} are satisfied, then $\widetilde{g}(\mathbf{x}) = g^0 = \underline{g}$. 
Since the integral over the cell of the derivatives of a periodic function is equal to zero, then
$$
\int_\Omega B_1(\boldsymbol{\theta}; \mathbf{D}) \Lambda(\mathbf{x}) \, d\mathbf{x} =0
$$
and hence, we have $g^{(1)} (\boldsymbol{\theta})=0$.  Consequently, by \eqref{G(theta)=}, \hbox{$G(\boldsymbol{\theta})=0$}.

Now, suppose that $n=1$ and the matrices $g(\mathbf{x})$, $b_\beta$, $|\beta|=p$, have real entries.
Then, for $p$ even, the solution   $\Lambda(\mathbf{x})$ of problem 
\eqref{Lambda_problem} is a $(1 \times m)$-matrix with real entries. Hence, the matrix 
$\widetilde{g}(\mathbf{x})= g(\mathbf{x}) (b(\mathbf{D}) \Lambda(\mathbf{x}) + \mathbf{1}_m)$
is an $(m \times m)$-matrix with real entries. Next,  $B_1(\boldsymbol{\theta}; \mathbf{D}) \Lambda(\mathbf{x})$
is an $(m \times m)$-matrix with imaginary entries. By \eqref{g1(theta)=},
$g^{(1)}(\boldsymbol{\theta})$ is a Hermitian $(m\times m)$-matrix with imaginary entries. 
Consequently,
$b(\boldsymbol{\theta})^* g^{(1)}(\boldsymbol{\theta})b(\boldsymbol{\theta})$ is equal to zero, as
a Hermitian imaginary $(1\times 1)$-matrix.  

For $p$ odd, the solution   $\Lambda(\mathbf{x})$ of problem 
\eqref{Lambda_problem} is a $(1 \times m)$-matrix with imaginary entries. The matrix 
$\widetilde{g}(\mathbf{x})$ has real entries. Next,  $B_1(\boldsymbol{\theta}; \mathbf{D}) \Lambda(\mathbf{x})$
is an $(m \times m)$-matrix with imaginary entries. Hence,
$g^{(1)}(\boldsymbol{\theta})$ is a Hermitian $(m\times m)$-matrix with imaginary entries. 
Again,
$b(\boldsymbol{\theta})^* g^{(1)}(\boldsymbol{\theta})b(\boldsymbol{\theta})$ is equal to zero, as
a Hermitian imaginary $(1\times 1)$-matrix.  
\end{proof}

\begin{rem}
In the general case, the operator $G(\boldsymbol{\theta})$ may be non-zero. 
In particular, it is easy to give examples of the scalar operator $A= b(\mathbf{D})^* g(\mathbf{x}) b(\mathbf{D})$ (i.~e., $n=1$), where $g(\mathbf{x})$ is a Hermitian matrix with complex entries, such that the corresponding operator $G(\boldsymbol{\theta})$ is not zero. 
\end{rem}

\subsection{Approximation for the operator exponential of $A(\mathbf{k})$}

In $L_2(\mathbb{R}^d; \mathbb{C}^n)$, consider the operator $H_0 := - \Delta$. 
Let $H_0(\mathbf{k})$ be the operator family in $L_2(\Omega;\mathbb{C}^n)$ corresponding to the operator $H_0$.
Then $H_0(\mathbf{k})$ is given by the differential expression 
$ |\mathbf{D} + \mathbf{k}|^2$ on the domain $\widetilde{H}^{2}(\Omega;\mathbb{C}^n)$. 
Denote
\begin{equation}
\label{op_R0}
R_0(\mathbf{k},\varepsilon) := \varepsilon^2 (H_0(\mathbf{k}) + \varepsilon^2 I)^{-1}. 
\end{equation}
Clearly, we have
\begin{equation}
\label{op_R0_identity}
R_0(\mathbf{k},\varepsilon)^{s/2} P := \varepsilon^s (t^2  + \varepsilon^2 )^{-s/2} P,\quad t=|\mathbf{k}|,\ \ s>0. 
\end{equation}

We apply  Theorem \ref{th_est3} to the operator family $A(t, \boldsymbol{\theta}) = A(\mathbf{k})$. 
Note that, by \eqref{eff_op_symbol}, \eqref{eff_op2}, 
$
e^{-i \tau t^{2p} S( \boldsymbol{\theta})P} P = e^{-i \tau A^0(\mathbf{k})} P.
$
Thus, the operator \eqref{J_def} turns into  
$$
J(\mathbf{k}, \tau) := \left( e^{-i \tau A(\mathbf{k})} - e^{-i \tau A^0(\mathbf{k})}\right)P.
$$
It remains to implement the values of the constants in estimates. The constants $\delta$ and $t_0$
are given by  \eqref{delta_fix} and \eqref{t0_fix}, respectively; they do not depend on $\boldsymbol{\theta}$.
The constant $C_0$ is given by \eqref{C0}. 
Using \eqref{Xp_est}, we can replace the constant 
$
C_T(\boldsymbol{\theta}) = p C_0^2 + \| X_p(\boldsymbol{\theta})\|^2 \delta^{-1}
$
depending now on $\boldsymbol{\theta}$ by
$
C_T= p C_0^2  + \alpha_1 \|g\|_{L_\infty} \delta^{-1}.
$
According to \eqref{1.5zz}, \eqref{1.6zz}, we put
$C_2= C(p) C_T$, $C_3= C(p) C_T^{2p+1}.$

Now, applying \eqref{exp_est1} and taking \eqref{op_R0_identity} into account, we obtain 
\begin{equation}
\label{exp_est1_DO}
\left\|  J( \mathbf{k}, \tau \varepsilon^{-2p}) R_0(\mathbf{k},\varepsilon)^{p+1/2} P 
\right\|_{L_2(\Omega) \to L_2(\Omega)}\! 
\leqslant \!( C_2  + C_3 |\tau|) \varepsilon, \quad \tau \in \mathbb{R}, \ \; \varepsilon >0,\ \; |\mathbf{k}| \leqslant t_0.
\end{equation}
For $\mathbf{k} \in \widetilde{\Omega}$, $|\mathbf{k}|>t_0$, estimates are trivial. 
Obviously, $\| J(\mathbf{k}, \tau)\| \leqslant 2$ and \hbox{$\|R_0(\mathbf{k},\varepsilon)\| \leqslant 1$}.
By \eqref{op_R0_identity},
\begin{equation*}
\left\|R_0(\mathbf{k},\varepsilon)^{1/2} P \right\|_{L_2(\Omega) \to L_2(\Omega)} \leqslant t_0^{-1} \varepsilon,
\quad \mathbf{k} \in \widetilde{\Omega}, \ \; |\mathbf{k}|>t_0,\ \; \varepsilon >0.
\end{equation*}
Hence,
\begin{equation}
\label{exp_est2_DO}
\left\|  J( \mathbf{k}, \tau \varepsilon^{-2p}) R_0(\mathbf{k},\varepsilon)^{p+1/2} P \right\|_{L_2(\Omega) \to L_2(\Omega)}  \! \leqslant \! 2 \, t_0^{-1} \varepsilon, \quad \tau \in \mathbb{R}, \ \,\mathbf{k} \in \widetilde{\Omega}, 
\ \,|\mathbf{k}|>t_0,\ \varepsilon >0.
\end{equation}
Finally, let us show that, within the margin of error, the projection $P$ in estimates \eqref{exp_est1_DO} and \eqref{exp_est2_DO} can be removed. Indeed, using the discrete Fourier transform, we have
\begin{equation}
\label{R0(I-P)}
\left\| R_0(\mathbf{k},\varepsilon)^{1/2}(I-P) \right\|_{L_2(\Omega) \to L_2(\Omega)} \! = \! \max_{0 \ne \mathbf{b} \in \widetilde{\Gamma}} 
\varepsilon (|\mathbf{b} + \mathbf{k}|^2 + \varepsilon^2)^{-1/2} \! \leqslant \! r_0^{-1} \varepsilon,
\quad \mathbf{k} \in \widetilde{\Omega}, \ \;\varepsilon >0.
\end{equation}
Consequently,
\begin{equation}
\label{exp_est3_DO}
\begin{aligned}
\left\|  \left( e^{-i \tau \varepsilon^{-2p} A(\mathbf{k})} - e^{-i \tau \varepsilon^{-2p} A^0(\mathbf{k})}\right)
 R_0(\mathbf{k},\varepsilon)^{p+1/2} (I-P) \right\|_{L_2(\Omega) \to L_2(\Omega)} 
\! \leqslant \! 2 \, r_0^{-1} \varepsilon, 
\\ 
\tau \in \mathbb{R}, \ \; \mathbf{k} \in \widetilde{\Omega}, \ \; \varepsilon >0.
\end{aligned}
\end{equation}

Combining \eqref{exp_est1_DO}, \eqref{exp_est2_DO}, and \eqref{exp_est3_DO}, we arrive at the following result.

\begin{thm}\label{th_A(k)}
For $\tau \in {\mathbb R}$, $\varepsilon >0$, and $\mathbf{k} \in \widetilde{\Omega}$
we have
\begin{equation*}
\left\|  \left( e^{-i \tau \varepsilon^{-2p} A(\mathbf{k})} - e^{-i \tau \varepsilon^{-2p} A^0(\mathbf{k})}\right)
 R_0(\mathbf{k},\varepsilon)^{p+1/2}  \right\|_{L_2(\Omega) \to L_2(\Omega)} 
\leqslant {\mathfrak C}_1 (1+|\tau|) \varepsilon.
\end{equation*}
The constant ${\mathfrak C}_1$ depends only on $d$, $p$, $\alpha_0$, $\alpha_1$, $\|g\|_{L_\infty}$, $\|g^{-1}\|_{L_\infty}$,
and the parameters of the lattice $\Gamma$. 
\end{thm}

Note that, by \eqref{op_R0_identity}, 
\begin{equation}
\label{R0P_t>t0_2}
\|R_0(\mathbf{k},\varepsilon) P \| \leqslant t_0^{-2} \varepsilon^2,
\quad \mathbf{k} \in \widetilde{\Omega}, \ \; |\mathbf{k}|>t_0,\ \; \varepsilon >0.
\end{equation}
Similarly to \eqref{R0(I-P)}, we have
\begin{equation}
\label{R0(I-P)_2}
\| R_0(\mathbf{k},\varepsilon) (I-P)\| = \max_{0 \ne \mathbf{b} \in \widetilde{\Gamma}} 
\varepsilon^2 (|\mathbf{b} + \mathbf{k}|^2 + \varepsilon^2)^{-1} \leqslant r_0^{-2} \varepsilon^2,
\quad \mathbf{k} \in \widetilde{\Omega}, \ \; \varepsilon >0.
\end{equation}
Applying Theorem \ref{th_est3_improved}
 and using \eqref{R0P_t>t0_2}, \eqref{R0(I-P)_2}, we arrive at the following result.

\begin{thm}\label{th_A(k)_G=0_2}
Let $G (\boldsymbol{\theta})$ be the operator given by  \eqref{G(theta)=}, \eqref{g1(theta)=}.
Suppose that $G (\boldsymbol{\theta})=0$ for any $\boldsymbol{\theta} \in \mathbb{S}^{d-1}$.
Then for $\tau \in {\mathbb R}$, $\varepsilon >0$, and $\mathbf{k} \in \widetilde{\Omega}$
we have
\begin{equation*}
\left\|  \left( e^{-i \tau \varepsilon^{-2p} A(\mathbf{k})} - e^{-i \tau \varepsilon^{-2p} A^0(\mathbf{k})}\right)
 R_0(\mathbf{k},\varepsilon)^{p+1}  \right\|_{L_2(\Omega) \to L_2(\Omega)} 
\leqslant {\mathfrak C}_{2} (1+|\tau|) \varepsilon^2.
\end{equation*}
The constant ${\mathfrak C}_{2}$ depends only on $d$, $p$, $\alpha_0$, $\alpha_1$, $\|g\|_{L_\infty}$, $\|g^{-1}\|_{L_\infty}$,
and the parameters of the lattice $\Gamma$. 
\end{thm}

\section{Approximation for the operator exponential of $A$}

Let $A$ be the operator in $L_2(\mathbb{R}^d;\mathbb{C}^n)$ given by 
$A= b(\mathbf{D})^* g(\mathbf{x}) b(\mathbf{D})$; see  \eqref{op_A}.  
Let $A^0 = b(\mathbf{D})^* g^0 b(\mathbf{D})$ be the effective operator \eqref{eff_op}.
Recall the notation $H_0 = - \Delta$ and denote
\begin{equation}
\label{ROeps}
R_0(\varepsilon):= \varepsilon^2 ( H_0 + \varepsilon^2 I)^{-1}.
\end{equation}

From expansion \eqref{direct2} it follows that
$$
e^{-i \tau \varepsilon^{-2p} A} = {\mathcal U}^{-1} \Bigl( \int_{\widetilde{\Omega}} \oplus 
e^{-i \tau \varepsilon^{-2p} A(\mathbf{k})} \, d \mathbf{k} \Bigr) {\mathcal U}.
$$
The operator $e^{-i \tau \varepsilon^{-2p} A^0}$ admits a similar expansion.
The operator \eqref{ROeps} is decomposed into the direct integral of the operators \eqref{op_R0}:
$$
R_0( \varepsilon) = {\mathcal U}^{-1} \Bigl( \int_{\widetilde{\Omega}} \oplus 
R_0 ( \mathbf{k}, \varepsilon)  \, d \mathbf{k} \Bigr) {\mathcal U}.
$$
From these direct integral expansions, taking into account that $\mathcal U$ is unitary, we obtain
$$
\begin{aligned}
&\left\|  \left( e^{-i \tau \varepsilon^{-2p} A} - e^{-i \tau \varepsilon^{-2p} A^0}\right)
 R_0(\varepsilon)^{s/2}  \right\|_{L_2(\mathbb{R}^d) \to L_2(\mathbb{R}^d)} 
 \\
 & = \sup_{\mathbf{k} \in \widetilde{\Omega}} \left\|  \left( e^{-i \tau \varepsilon^{-2p} A(\mathbf{k})} - e^{-i \tau \varepsilon^{-2p} A^0(\mathbf{k})}\right)
 R_0(\mathbf{k},\varepsilon)^{s/2}  \right\|_{L_2(\Omega) \to L_2(\Omega)}.
\end{aligned}
$$
Combining this with Theorem \ref{th_A(k)}, we arrive at the following result.

\begin{thm}\label{th_A}
For $\tau \in {\mathbb R}$ and $\varepsilon >0$ we have
\begin{equation}
\label{exp_est4_DO}
\left\|  \left( e^{-i \tau \varepsilon^{-2p} A} - e^{-i \tau \varepsilon^{-2p} A^0}\right)
 R_0(\varepsilon)^{p+1/2}  \right\|_{L_2(\mathbb{R}^d) \to L_2(\mathbb{R}^d)} 
\leqslant {\mathfrak C}_1 (1+|\tau|) \varepsilon.
\end{equation}
\end{thm}

 Similarly,  Theorem  \ref{th_A(k)_G=0_2} implies the following result.
 
\begin{thm}\label{th_A_G=0_2}
Let $G (\boldsymbol{\theta})$ be the operator given by  \eqref{G(theta)=}, \eqref{g1(theta)=}.
Suppose that $G (\boldsymbol{\theta})=0$ for any $\boldsymbol{\theta} \in \mathbb{S}^{d-1}$.
Then for $\tau \in {\mathbb R}$ and $\varepsilon >0$
we have
\begin{equation*}
\left\|  \left( e^{-i \tau \varepsilon^{-2p} A} - e^{-i \tau \varepsilon^{-2p} A^0}\right)
 R_0(\varepsilon)^{p+1}  \right\|_{L_2(\mathbb{R}^d) \to L_2(\mathbb{R}^d)} 
\leqslant {\mathfrak C}_2 (1+|\tau|) \varepsilon^2.
\end{equation*}
\end{thm}

\section{Homogenization of the Schr{\"o}dinger-type equation}

\subsection{The operator $A_\varepsilon$. The scaling transformation}

We use the notation $g^\varepsilon(\mathbf{x}):= g(\varepsilon^{-1} \mathbf{x})$, $\varepsilon>0$.
Our main object is the operator $A_\varepsilon$ acting in $L_2(\mathbb{R}^d;\mathbb{C}^n)$ and formally given by
\begin{equation}
\label{op_Aeps}
A_\varepsilon = b(\mathbf{D})^* g^\varepsilon(\mathbf{x}) b(\mathbf{D}).
\end{equation}
The precise definition of $A_\varepsilon$ is given in terms of the corresponding quadratic form; cf. Subsection \ref{sec_operator}.
Our goal is to  approximate  the operator exponential $e^{-i \tau A_\varepsilon}$ for small~$\varepsilon$.

Let $T_\varepsilon$ be the scaling transformation defined by 
$(T_\varepsilon \mathbf{u})(\mathbf{x})= \varepsilon^{d/2} \mathbf{u} (\varepsilon \mathbf{x}).$  
Then $T_\varepsilon$ is unitary in $L_2(\mathbb{R}^d;\mathbb{C}^n)$. 
We have $A_\varepsilon = \varepsilon^{-2p} T_\varepsilon^* A T_\varepsilon$. Hence,
$
e^{-i \tau A_\varepsilon} =  T_\varepsilon^* e^{-i \tau \varepsilon^{-2p}  A} T_\varepsilon.
$
A~similar relation holds for the effective operator $A^0$:
$
e^{-i \tau A^0} =  T_\varepsilon^* e^{-i \tau \varepsilon^{-2p}  A^0} T_\varepsilon.
$
Applying the scaling transformation to the resolvent of $H_0 = -\Delta$, we obtain
$$
(H_0 + I)^{-1} = \varepsilon^2  T_\varepsilon^* (H_0 + \varepsilon^2 I)^{-1}  T_\varepsilon =
 T_\varepsilon^* R_0(\varepsilon)  T_\varepsilon.
$$
Using these relations and taking into account that $T_\varepsilon$ is unitary, we have
\begin{equation}
\label{exp_scaling}
\begin{aligned}
&\Bigl\| \left( e^{-i \tau A_\varepsilon} - e^{-i \tau A^0} \right) (H_0 +I)^{-s/2} 
\Bigr\|_{L_2(\mathbb{R}^d) \to L_2(\mathbb{R}^d)} 
\\
&=
\Bigl\| \left( e^{-i \tau \varepsilon^{-2p}  A}  - e^{-i \tau \varepsilon^{-2p} A^0} \right) 
R_0(\varepsilon)^{s/2}  \Bigr\|_{L_2(\mathbb{R}^d) \to L_2(\mathbb{R}^d)},\ \ s>0.
\end{aligned}
\end{equation}

\subsection{Approximation for the operator exponential $e^{-i \tau A_\varepsilon}$}

Combining \eqref{exp_scaling} and \eqref{exp_est4_DO}, we see that
\begin{equation*}
\Bigl\| \left( e^{-i \tau A_\varepsilon} - e^{-i \tau A^0} \right) (H_0 +I)^{-(p+1/2)} \Bigr\|_{L_2(\mathbb{R}^d) \to L_2(\mathbb{R}^d)} 
\! \leqslant  \! {\mathfrak C}_1 (1+|\tau|) \varepsilon,\quad \tau \in \mathbb{R},\ \varepsilon >0.
\end{equation*}
Since $(H_0+I)^{p+1/2}$ is an isometric isomorphism of  
the Sobolev space $H^{2p+1}(\mathbb{R}^d;\mathbb{C}^n)$ onto $L_2(\mathbb{R}^d;\mathbb{C}^n)$, this yields the following result.

\begin{thm}\label{th_Aeps}
For $\tau \in {\mathbb R}$ and $\varepsilon >0$ we have
\begin{equation}
\label{hom_est5_DO}
\left\|   e^{-i \tau A_\varepsilon} - e^{-i \tau  A^0}
   \right\|_{H^{2p+1}(\mathbb{R}^d) \to L_2(\mathbb{R}^d)} 
\leqslant {\mathfrak C}_1 (1+|\tau|) \varepsilon.
\end{equation}
The constant ${\mathfrak C}_1$ depends only on $d$, $p$, $\alpha_0$, $\alpha_1$, $\|g\|_{L_\infty}$, $\|g^{-1}\|_{L_\infty}$,
and the parameters of the lattice $\Gamma$. 
\end{thm}

Interpolating between the obvious estimate
$
\bigl\|   e^{-i \tau A_\varepsilon} - e^{-i \tau  A^0}
   \bigr\|_{L_2(\mathbb{R}^d) \to L_2(\mathbb{R}^d)} \leqslant 2 
$
and \eqref{hom_est5_DO}, we obtain the following corollary.

\begin{corollary}\label{cor1}
Let $0\leqslant s \leqslant 2p+1$. For $\tau \in {\mathbb R}$ and $\varepsilon >0$ we have
\begin{equation*}
\left\|   e^{-i \tau A_\varepsilon} - e^{-i \tau  A^0}
   \right\|_{H^{s}(\mathbb{R}^d) \to L_2(\mathbb{R}^d)} 
\leqslant C(s) (1+|\tau|)^{s/(2p+1)} \varepsilon^{s/(2p+1)}.
\end{equation*}
Here $C(s) = 2^{1-s/(2p+1)} {\mathfrak C}_1^{s/(2p+1)}$.
\end{corollary}

Similarly, by the scaling transformation, we deduce the following result from Theorem \ref{th_A_G=0_2}.

\begin{thm}\label{th_Aeps_G=0_2}
Let $G (\boldsymbol{\theta})$ be the operator given by  \eqref{G(theta)=}, \eqref{g1(theta)=}.
Suppose that $G (\boldsymbol{\theta})=0$ for any $\boldsymbol{\theta} \in \mathbb{S}^{d-1}$.
Then for $\tau \in {\mathbb R}$ and $\varepsilon >0$ we have
\begin{equation*}
\left\|   e^{-i \tau A_\varepsilon} - e^{-i \tau  A^0}
   \right\|_{H^{2p+2}(\mathbb{R}^d) \to L_2(\mathbb{R}^d)} 
\leqslant {\mathfrak C}_2 (1+|\tau|) \varepsilon^2.
\end{equation*}
The constant ${\mathfrak C}_2$ depends only on $d$, $p$, $\alpha_0$, $\alpha_1$, $\|g\|_{L_\infty}$, $\|g^{-1}\|_{L_\infty}$,
and the parameters of the lattice $\Gamma$. 
\end{thm}

\begin{corollary}\label{cor3}
Under the assumptions of Theorem \emph{\ref{th_Aeps_G=0_2}}, let $0\leqslant s \leqslant 2p+2$. For $\tau \in {\mathbb R}$ and $\varepsilon >0$ we have
\begin{equation*}
\left\|   e^{-i \tau A_\varepsilon} - e^{-i \tau  A^0}
   \right\|_{H^{s}(\mathbb{R}^d) \to L_2(\mathbb{R}^d)} 
\leqslant C'(s) (1+|\tau|)^{s/(2p+2)} \varepsilon^{s/(p+1)}.
\end{equation*}
Here $C'(s) = 2^{1-s/(2p+2)} {\mathfrak C}_2^{s/(2p+2)}$. In particular, for $s =p+1$
\begin{equation}
\label{hom_est26_DO_p+1}
\left\|   e^{-i \tau A_\varepsilon} - e^{-i \tau  A^0}
   \right\|_{H^{p+1}(\mathbb{R}^d) \to L_2(\mathbb{R}^d)} 
\leqslant C'(p+1) (1+|\tau|)^{1/2} \varepsilon.
\end{equation}
\end{corollary}

Recall that some sufficient conditions ensuring that $G (\boldsymbol{\theta})\equiv 0$ are given in Proposition \ref{prop_cond_G=0}.

\begin{rem}
Theorem \ref{th_Aeps_G=0_2} shows that, if $G (\boldsymbol{\theta})\equiv 0$,
the difference $e^{-i \tau A_\varepsilon} - e^{-i \tau  A^0}$ is of order $O(\varepsilon^2)$ in a suitable norm.
Estimate \eqref{hom_est26_DO_p+1} improves \eqref{hom_est5_DO} regarding both 
the norm type and the dependence of the estimate on $\tau$.
We note that for the second-order operators there is analog of  estimate \eqref{hom_est26_DO_p+1}  (see \cite{D}), but there is no analog of Theorem \ref{th_Aeps_G=0_2}.
\end{rem}

\subsection{Homogenization of the Cauchy problem for the  Schr{\"o}dinger-type equation}

Let $\mathbf{u}_\varepsilon(\mathbf{x}, \tau)$ be the solution of the following Cauchy problem:
\begin{equation}
\label{Cauchy}
\begin{aligned}
& i \,\partial_\tau \mathbf{u}_\varepsilon(\mathbf{x}, \tau) = b(\mathbf{D})^* g^\varepsilon(\mathbf{x}) b(\mathbf{D}) 
\mathbf{u}_\varepsilon(\mathbf{x}, \tau) + \mathbf{F}(\mathbf{x}, \tau),\quad \mathbf{x}\in \mathbb{R}^d,\ \;
 \tau \in \mathbb{R},
 \\
 & \mathbf{u}_\varepsilon(\mathbf{x}, 0) = \boldsymbol{\phi} (\mathbf{x}),\quad \mathbf{x}\in \mathbb{R}^d,
\end{aligned}
\end{equation}
where $\boldsymbol{\phi} \in L_2(\mathbb{R}^d; \mathbb{C}^n)$ and
$\mathbf{F} \in L_{1,\text{loc}}( \mathbb{R};L_2(\mathbb{R}^d; \mathbb{C}^n))$.
The solution of problem \eqref{Cauchy} admits the following representation:
\begin{equation}
\label{Cauchy2}
\mathbf{u}_\varepsilon(\cdot, \tau) = e^{-i \tau A_\varepsilon}  \boldsymbol{\phi} - i \int_0^\tau
e^{-i (\tau - \widetilde{\tau}) A_\varepsilon} \mathbf{F}(\cdot, \widetilde{\tau})\, d\widetilde{\tau}.
\end{equation}

Let $\mathbf{u}_0(\mathbf{x}, \tau)$ be the solution of the homogenized Cauchy problem:
\begin{equation}
\label{Cauchy_eff}
\begin{aligned}
& i \,\partial_\tau \mathbf{u}_0(\mathbf{x}, \tau) = b(\mathbf{D})^* g^0 b(\mathbf{D}) 
\mathbf{u}_0(\mathbf{x}, \tau) + \mathbf{F}(\mathbf{x}, \tau),\quad \mathbf{x}\in \mathbb{R}^d,\ \;
 \tau \in \mathbb{R},
 \\
 & \mathbf{u}_0 (\mathbf{x}, 0) = \boldsymbol{\phi} (\mathbf{x}),\quad \mathbf{x}\in \mathbb{R}^d.
\end{aligned}
\end{equation}
The solution of problem \eqref{Cauchy_eff} can be represented as
\begin{equation}
\label{Cauchy_eff2}
\mathbf{u}_0(\cdot, \tau) = e^{-i \tau A^0}  \boldsymbol{\phi} - i \int_0^\tau
e^{-i (\tau - \widetilde{\tau}) A^0} \mathbf{F}(\cdot, \widetilde{\tau})\, d\widetilde{\tau}.
\end{equation}

\begin{thm}\label{th_Cauchy}
Let $\mathbf{u}_\varepsilon(\mathbf{x}, \tau)$ be the solution of the Cauchy problem \eqref{Cauchy}.
Let $\mathbf{u}_0(\mathbf{x}, \tau)$ be the solution of the homogenized problem \eqref{Cauchy_eff}.

\noindent $1^\circ$. Let $0 \leqslant s \leqslant 2p+1$. 
If $\boldsymbol{\phi} \in H^{s}(\mathbb{R}^d; \mathbb{C}^n)$ and
\emph{$\mathbf{F} \in L_{1,\text{loc}}( \mathbb{R};H^{s}(\mathbb{R}^d; \mathbb{C}^n))$}, then 
for $\tau \in {\mathbb R}$ and $\varepsilon >0$ we have
\begin{equation}
\label{Cauchy_interp0}
\begin{aligned}
&\left\| \mathbf{u}_\varepsilon(\cdot, \tau) - \mathbf{u}_0(\cdot, \tau) \right\|_{L_2(\mathbb{R}^d)} \leqslant
C(s) (1+ |\tau|)^{s/(2p+1)} \varepsilon^{s/(2p+1)}
\\
 &\times \left( \| \boldsymbol{\phi} \|_{H^{s}(\mathbb{R}^d)} + 
\| \mathbf{F} \|_{L_1((0,\tau);H^{s}(\mathbb{R}^d))} \right).
\end{aligned}
\end{equation}

\noindent $2^\circ$.  
If $\boldsymbol{\phi} \in L_2(\mathbb{R}^d; \mathbb{C}^n)$ and
\emph{$\mathbf{F} \in L_{1,\textrm{loc}}( \mathbb{R};L_2(\mathbb{R}^d; \mathbb{C}^n))$}, then 
for $\tau \in {\mathbb R}$  we have
$$
\lim_{\varepsilon \to 0}\left\| \mathbf{u}_\varepsilon(\cdot, \tau) - \mathbf{u}_0(\cdot, \tau) \right\|_{L_2(\mathbb{R}^d)} =0.
$$
\end{thm}

\begin{proof}
Statement $1^\circ$ follows from Corollary \ref{cor1} and representations \eqref{Cauchy2}, \eqref{Cauchy_eff2}.

Estimate \eqref{Cauchy_interp0} with $s=0$ means that
the norm $\left\| \mathbf{u}_\varepsilon(\cdot, \tau) - \mathbf{u}_0(\cdot, \tau) \right\|_{L_2(\mathbb{R}^d)}$
is uniformly bounded
provided that $\boldsymbol{\phi} \in L_2(\mathbb{R}^d; \mathbb{C}^n)$ and
$\mathbf{F} \in L_{1,\textrm{loc}}( \mathbb{R};L_2(\mathbb{R}^d; \mathbb{C}^n))$.
Hence, using  statement $1^\circ$ with $s=2p+1$ and applying the  Banach--Steinhaus theorem, we obtain 
statement $2^\circ$.
\end{proof}

Similarly, from Corollary \ref{cor3} we deduce the following result.

\begin{thm}\label{th_Cauchy3}
Let $G (\boldsymbol{\theta})$ be the operator given by  \eqref{G(theta)=}, \eqref{g1(theta)=}.
Suppose that \hbox{$G (\boldsymbol{\theta})=0$} for any $\boldsymbol{\theta} \in \mathbb{S}^{d-1}$.
Let $\mathbf{u}_\varepsilon(\mathbf{x}, \tau)$ be the solution of the Cauchy problem \eqref{Cauchy}.
Let $\mathbf{u}_0(\mathbf{x}, \tau)$ be the solution of the homogenized problem \eqref{Cauchy_eff}.
 Let $0 \leqslant s \leqslant 2p+2$. 
If $\boldsymbol{\phi} \in H^{s}(\mathbb{R}^d; \mathbb{C}^n)$ and
\emph{$\mathbf{F} \in L_{1,\text{loc}}( \mathbb{R};H^{s}(\mathbb{R}^d; \mathbb{C}^n))$}, then 
for $\tau \in {\mathbb R}$ and $\varepsilon >0$ we have
\begin{equation}
\label{Cauchy_interp}
\begin{aligned}
&\left\| \mathbf{u}_\varepsilon(\cdot, \tau) - \mathbf{u}_0(\cdot, \tau) \right\|_{L_2(\mathbb{R}^d)} \leqslant
C'(s) (1+ |\tau|)^{s/(2p+2)} \varepsilon^{s/(p+1)}
\\
& \times
\left( \| \boldsymbol{\phi} \|_{H^{s}(\mathbb{R}^d)} + 
\| \mathbf{F} \|_{L_1((0,\tau);H^{s}(\mathbb{R}^d))} \right).
\end{aligned}
\end{equation}
\end{thm}

\begin{acknowledgments}
This research was  supported by the Russian Science Foundation (grant no. 17-11-01069).
\end{acknowledgments}

\end{document}